\documentclass[12pt,a4paper]{article}
\usepackage[english]{babel}
\usepackage[utf8]{inputenc} 
\usepackage{amssymb}
\usepackage[T1]{fontenc}
\usepackage{amsmath}
\usepackage{amsthm}
\usepackage{lipsum}
\newtheorem{theorem}{Theorem}
\newtheorem{lemme}{Lemma}

\newtheorem{Cor}{Corollary}
\newtheorem{proposition}{Proposition}
\newtheorem{remarque}{Remark}

\newcommand{\dd}{\mathrm{d}} 

\newcommand{\F}{\mathcal{F}}
\newcommand{\R}{\mathbb{R}}
\newcommand{\Z}{\mathbb{Z}}
\newcommand{\N}{\mathbb{N}}

\renewcommand{\P}{\mathbb{P}}

\newcommand{\Lp}{\mathcal{L}p}

\newcommand{\sL}{\mathbb{L}}

\renewcommand{\SS}{\mathbb{S}}

\newcommand{\bx}{\mathbf{x}}
\newcommand{\by}{\mathbf{y}}
\newcommand{\bX}{\mathbf{X}}

\DeclareMathOperator\Var{Var}

\DeclareMathOperator\E{E}
\DeclareMathOperator\MSE{MSE}

\newcounter{tictac}

\newenvironment{fleuve}{
   \begin{list}{$\textbf{\emph{\arabic{tictac})}}$ }{\usecounter{tictac} \leftmargin 1cm\labelwidth
2em}}{\end{list}}

\def\1{\,\rlap{\mbox{\small\rm 1}}\kern.15em 1}

\def\build#1_#2^#3{\mathrel{\mathop{\kern 0pt#1}\limits_{#2}^{#3}}}
\def\tend#1#2{\build\hbox to 12mm{\rightarrowfill}_{#1\rightarrow #2}^{a.s.}}

\def\converge#1#2#3{\build\hbox to
15mm{\rightarrowfill}_{#1\rightarrow #2}^{\hbox{\scriptsize #3}}}

\title{Weak and Strong Consistency of non-parametric estimate of potential function for stationary and isotropic pairwise interaction point process}
\author{Nadia Morsli\\
nadia.morsli@imag.fr\\
Laboratory Jean Kuntzmann, Grenoble University, France.}
\date{}
\begin{document}
\maketitle

{\renewcommand\abstractname{Abstract}
\begin{abstract}
A method is proposed for estimating the potential function of a non-parametric estimator for stationary and isotropic pairwise interaction point process. 
The relation between a pair potential and the corresponding Papangelou conditional intensity is considered. Consistency
 and strong consistency of non-parametric estimate are proved in case of finite-range interaction potential.
\end{abstract}

\textbf{Keywords:}  
Non parametric estimation, kernel-type estimator,  pairwise interaction point 
process, Papangelou conditional intensity, consistency, rates of strong uniform consistency.
\thispagestyle{empty}
\newpage
\section{Introduction}

Gibbs point processes are a natural class of models for point patterns exhibiting interactions between the points.
 By far the most widely applied form in practical analysis is that of pairwise interaction, where the scale and 
strength of interaction between two points are determined by a so-called pair potential function. For a stationary 
and isotropic process the pair potential is a function of the distance between the two points.
Fields of applications for point processes are  image processing, analysis of the structure of tissues in medical sciences, 
 forestry (Mat\'ern \cite{Matern-B86}), ecology (Diggle \cite{B-Dig03}), 
spatial epidemiology (Lawson \cite{Lawson-A.B01})and  astrophysics (Neyman and Scott \cite{J-Neyman58}).

Pairwise interaction point process densities are intractable as the normalizing constant is unknown and/or extremely complicated to approximate. However, we
can resort to estimates of parameters using the conditional intensity.
In this paper, we suggest a new non-parametric  estimate 
of the pair potential function for stationary and isotropic pairwise interaction point process specified by a 
Papangelou conditional intensity  on increasing regions single realization is observed. In this cas a point process
 is defined as a random locally-finite counting measure on the $d$-dimensional Euclidean space $\R^d$ . 
Consistency and strong consistency of the resulting estimator  are established.

To our knowledge only one attempt to solve the problem of non-parametric 
estimation of the pair correlation function and its
approximate relation to the pair potential through the Percus Yevick equation 
(Diggle et al. \cite{A-DigGatSti87}). The approximation is a result of a cluster expansion method, and it is 
accurate only for sparse data. 
Many attempts have been tried to estimate the
potential function from point pattern data in a parametric framework ;
maximization of likelihood approximations (Ogata and Tanemura \cite{Ogata81}, Ogata and Tanemura \cite{Ogata84}, 
Penttinen \cite{A-Penttien84}), pseudolikelihood maximization (Besag et al. \cite{Besag82}, Jensen and M{\o}ller \cite{A-JenMol91})
 and also some ad hoc methods  (Strauss \cite{Strauss75}, Ripley \cite{A-RipKel77}, 
Hanisch and Stoyan \cite{Hanishch86},  Diggle and Gratton \cite{Gratton84}, 
Fiksel \cite{T.Fiksel2}, Takacs \cite{R.Takacs}, Billiot and  Goulard \cite{Billiot-Goulard2001}).

Our paper is organized as follows. Section \ref{sub} introduces basic notation and definitions.
 In Section \ref{examples}, we briefly present some models satisfying the assumptions needed to prove our asymptotic results.
In Section \ref{result}, we present our main results. Consistency of non-parametric estimator is proved in Section \ref{mean}, it is based on the knowledge of Papangelou conditional intensity and 
the iterated Georgii-Nguyen-Zessin formula. Using Orlicz spaces we can obtain a  strong consistency of non-parametric estimator 
in Section \ref{rates}.
\section{Basic notation and definitions}\label{sub}
Throughout the paper we adopt the following notation.
We denote the space of locally finite  point configurations in $\R^d$ by $N_{lf}$.
The volume of a bounded Borel set $W$ of $\R^d$ is denoted by $|W|$ and
$o$ denotes the origin.  For all finite subset ${\Gamma}$ of $\Z^d$, we denote  $|{\Gamma}|$ the number of elements in $\Gamma$. $||\cdot||$ denotes Euclidean distance on $\R^d$.
$\sigma_d=\frac{2\pi^{d/2}}{\Gamma(d/2)}$ is the measure of the unit sphere in $\R^d$. Let $\SS^{d-1}$  be the unit sphere in $\R^d$.

Papangelou conditional intensity (M{\o}ller and Waagepetersen \cite{B-MolWaa04}) of pairwise interaction point process has the form
 \[
 \lambda(u,\bx)=\gamma_0(u)\exp\bigg({-\sum_{v\in \bx\backslash u}\gamma_0(\{u,v\})}\bigg).
 \]
If $\gamma_0(u)=\beta$ is a constant and $\gamma_0(\{u , v \}) =\gamma(||u-v||)$ is invariant under translations 
and rotations, then a pairwise interaction point process is said to be stationary and isotropic or homogeneous. 
 The  Papangelou conditional intensity can be interpreted as 
follows: for any $u\in \R^d$ and $\bx \in N_{lf}$, $\lambda(u,\bx)\dd u$ corresponds to the conditional probability of observing a point in a ball of volume $\dd u$ around $u$ given the rest of the point process is $\bx$. 
Fortunately does not contain a normalising factor.
\paragraph{}
  For convenience, throughout in this paper, we consider stationary and isotropic pairwise interaction point process. Then
its Papangelou conditional intensity at a location $u$ is given by
 \begin{equation}\label{modelle}
\lambda(u,\bx)= \beta^\star \exp\bigg({-\sum_{v\in \bx\backslash u}\gamma(||v-u ||)}\bigg), \ \forall u\in \R^d,  \bx\in N_{lf}
 \end{equation}
where $\beta^\star$ is the true value of the Poisson intensity parameter
, $\gamma$ is called the pair potential, a name that originates in physics: it measures the potential energy caused
by the interaction among pairs of points $(u,v)$ as a function of their distance $||v-u||$.
Usually a finite range of interaction, $R$, is assumed such that 
\begin{equation}\label{RG}
\gamma(||v-u||)=0 \quad \mbox {whenever} \  ||v-u||>R.
\end{equation}
We assume that $\gamma(||v-u||)>0$ for $||v-u|| \leq R$, so that typical realizations will be more or less regular compared to a 
completely random arrangement.
The pairwise interaction between points may also be described in terms of
the pair potential function $\gamma$ into the interaction function $\varPhi=\exp({-\gamma}).$
 For $\varPhi > 1$, $\lambda(u,\bx)$ is increasing in $\bx$. For $\varPhi < 1$, $\lambda(u,\bx)$ is decreasing in $\bx$ 
(the repulsive case). It can be computed for the case $\varPhi = 1$  which corresponds to the homogeneous Poisson point process with
with intensity $\beta^\star$.
\section{Examples of Papangelou conditional intensity}\label{examples}
Examples of conditional intensities are presented in  Baddeley et al \cite{A-BadTurMolHaz05}, M{\o}ller and
Waagepetersen (\cite{B-MolWaa04}, \cite{A-MolWaa07}). The
following presents some examples which have been applied in various contexts and  satisfying the assumptions needed to prove our asymptotic results.
\begin{enumerate}

\item A special case of pairwise interaction is the Strauss process. It has Papangelou conditional intensity
\[
 \lambda(u,\bx)=\beta\varPhi^{ n_{[0,R]}(u,\bx \setminus u)} 
 \]
where $\beta> 0, 0 \leq \varPhi \leq 1$ and  $n_{[0,R]}(u,\bx)= \displaystyle\sum_{v\in \bx} 1\!\!1 ( \|v-u\|\leq R)$
is the number of pairs in $\bx$ with distance not greater than $R$.

\item 
 Piecewise Strauss point process. 
\[
  \lambda(u,\bx)=\beta\prod_{j=1}^p  \varPhi_j ^{n_{[R_{j-1},R_j]}(u,\bx \setminus u)}
\]
 where $\beta>0$, $0\leq \varPhi_j \leq 1, n_{[R_{j-1},R_j]}(u,\bx)=\displaystyle\sum_{v\in \bx} 1\!\!1(\|v-u\|\in [R_{j-1},R_j])$ and $R_0=0<R_1<\ldots<R_p=R<\infty$.
 
\item Triplets point process.  
\[
 \lambda(u,\bx)=\beta\varPhi^{s_{[0,R]}(\bx\cup u)-s_{[0,R]}(\bx\setminus u)}
\]
 where $\beta>0$, $0\leq\varPhi\leq 1$ and $s_{[0,R]}(\bx)$ is the number of unordered triplets that are closer than $R$. 
\item
 Lennard-Jones model
\[
  \lambda(u,\bx)= \beta \prod_{v\in \bx \setminus u} \varPhi(\| v-u\|\big )
  \]
 with $\log \varPhi(r)=  \big( \theta^6 r^{-6}- \theta^{12}r^{-12}  \big)1\!\!1_{]0,R]}(r)$, for  $r=||v-u||$, where
 $\theta>0$ and $\beta>0$ ares parameters.
\end{enumerate}

\section{Main results}\label{result}

 Suppose that a single realization $\bx$ of a point process $\bX$ is observed in a bounded  
 window $W_n$ where ($W_n )_{n\geq 1}$ is a sequence of cubes growing up to $\R^d$.
  Throughout in this paper, $\widetilde{h}$ is a non-negative measurable function 
defined for all $u\in \R^d$, $ \bx\in N_{lf}$ by
\[
\widetilde{h}(u,\bx)=1\!\!1\left( \inf_{v\in \bx}||v-u||>  R\right)= 1\!\!1\left(d(u,\bx)> R \right),
\] 
note that
\[
\widetilde{F}(o,rv)=\E[\widetilde{h}(o,\bx)\widetilde{h}(rv,\bx)]
\]
and 
\[
 J(r)=\int_{\SS^{d-1}}\widetilde{F}(o,rv)\dd v.
\] 
 
 \paragraph{}
To estimate the function $\beta^{\star2}  J(r)\varPhi(r)$,  we introduce edge-corrected kernel-type estimator $\widehat{R}_n(r)$  defined by
 {\small
\begin{equation}\label{range24}
 \widehat{R}_n(r)= 
\frac{1}{b_n|W_{n\ominus 2R}|\sigma_d}
\sum_{\begin{subarray}{c} u,v \in \bX \\ ||v-u||\leq  R\end{subarray}}^{\neq} \!\!\!\!\!
\frac{1\!\!1_{W_{n\ominus 2R}}(u)}{||v-u||^{d-1}}\widetilde{h}(u,\bX \setminus \{u,v\})\widetilde{h}(v,\bX\setminus \{u,v\}) K_1\bigg(\frac{||v-u||-r}{b_n}\bigg).
\end{equation}
}
$\ominus$ will denote Minkowski substraction,  with the convention that
 
$
W_{n\ominus 2R}= W_{n}\ominus B(u,2R)=\{u\in W_n: ||u-v||\leq 2R  \quad \mbox {for all} \quad v\in W_n\}
$
denotes the $2R$-interior of the cubes $W_n$, with Lebesgue measure  $|W_{n\ominus 2R}|>0$.
 $\sum^{\neq}$ signifies summation over distinct pairs.
$K_1: \R\rightarrow \R$ is an univariate kernel function
 associated with a sequence $(b_n)_{n\geq 1}$ of  bandwidths  satisfying the following:
\paragraph{}
\textbf{Condition $K(1,\alpha):$} 
The sequence of bandwidths  $b_n>0$ for $n\geq 1$,  is chosen such that

\[\lim_{n\rightarrow\infty} b_n=0 \quad \mbox {and} \quad  \lim_{n\rightarrow\infty}b_n|W_{n\ominus 2R}|=\infty.\]
 The kernel function $K_1: \R\longrightarrow \R$ is non-negative and bounded with bounded support, such that:
\begin{equation*}
\int_\R K_1(u) \dd u=1, \int_\R u^{j}K_1(u)\dd u=0, j=0, 1, . . . , \alpha-1, \quad \mbox{for} \ \alpha\geq 2.
\end{equation*}
To estimate the function $\beta^\star J(r)$  we introduce
 empiric estimator $\widehat{J}_n(r)$  defined by
\begin{equation}\label{estima25}
  \widehat{J}_n(r)= \frac{1}{|W_{n\ominus 2R}|} \sum_{u \in \bX} 1\!\!1_{W_{n\ominus 2R}}(u)\widetilde{h}(u,\bX \setminus \{u\})h^\star (u,\bX\setminus \{u\}), 
\end{equation}
where $h^\star(u,\bx)= \int_{\SS^{d-1}}\widetilde{h}(rv-u,\bx)\dd v$. 
Using the spatial ergodic theorem of Nguyen and Zessin \cite{A-NguZes79}, estimator \eqref{estima25} turn out to be unbiased and strongly consistent. 
The natural estimator of  Poisson intensity $\beta^{\star}$ is
\begin{equation}\label{eq:defEst}
\widehat{\beta}_n=\frac{\sum_{u \in \bX} 1\!\!1_{\Lambda_{n,R}}(u)\widetilde{h}(u,\bX \setminus \{u\})}
{\int_{\Lambda_{n,R}}\widetilde{h}(u,\bX)\dd u}.
\end{equation}
This estimator turn out to be unbiased and strongly consistent and results on asymptotic normality were obtained by
 Morsli et al. \cite{JFMN}.

Plugging in the above estimator \eqref{estima25} and \eqref{eq:defEst},
then the interaction function $\varPhi(r)=\exp(-\gamma(r))$ for $r\in (0,R]$ can be estimated using edge-corrected 
non-parametric estimate   by
\begin{equation}\label{range23}
 \widehat{\varPhi}_n(r)=\frac{\widehat{R}_n(r)}{\widehat{\beta}_n\widehat{J}_n(r)}.
\end{equation}
The strong consistency of the estimators \eqref{estima25} and \eqref{eq:defEst} implies the following:
\begin{proposition}
Let $\gamma$ be pairwise interaction potential
defined in  \eqref{modelle} satisfying condition \eqref{RG}.
Let $K_1$ kernel function satisfying Condition $K(1,\alpha)$ and the function $ J(r)exp(-\gamma(r))$  has bounded 
and continuous partial derivatives of order $\alpha$ for all $\alpha\geq 1$
in $(r-\delta, r+\delta)$ for some $\delta > 0$. Then
as $n\rightarrow \infty$
\[
\widehat{\varPhi}_n(r){\longrightarrow} \exp(-\gamma(r)) \quad\textrm{in probability}\quad \P \quad\textrm{(resp.} \P \textrm{-a.s.)} \quad\textrm{iff}
\]
\[
\widehat{R}_n(r){\longrightarrow} \beta^{\star^2}J(r)\exp(-\gamma(r)) \quad\textrm{in probability}\quad \P \quad\textrm{(resp.} \P \textrm{- a.s.)}.
\]
\end{proposition}
The convergence in probability (consistency) for a wide class of point process will be discussed in Section \ref{mean}.
Conditions ensuring  uniform $\P$-a.s. convergence
of kernel-type estimator of $\widehat{R}_n(r)$ and the strong consistency $\widehat{\varPhi}_n(r)$ will be discussed in Section \ref{rates}.
\section{Consistoncy}\label{mean}
\subsection{Asymptotic behaviour mean squared error of the kernel-type estimator}
In this section we will derive bounds for the  mean squared error of the kernel estimator kernel-type estimator of 
 $\widehat{R}_n(r)$. We consider the mean square error of $\widehat{R}_n(r)$, 
$\MSE\big(\widehat{R}_n(r)\big)=\Var\big(\widehat{R}_n(r)\big) + \big(Biais(\widehat{R}_n(r)\big)^2$.
So convergence in $\MSE$ implies that as $n\rightarrow\infty$
$
\big(Biais(\widehat{R}_n(r)\big)^2=\big(\E \widehat{R}_n(r)-\beta^{\star 2}J(r)\exp(-\gamma(r))\big)^2\longrightarrow 0
$
and
$
\Var\big(\widehat{R}_n(r)\big)=\E\big(\widehat{R}_n(r)-\E\widehat{R}_n(r)\big)^2\longrightarrow 0.
$
Hence, $\widehat{R}_n(r)$ is consistent in the quadratic mean and hence consistent. 
For doing this, we first determine the asymptotic behaviour of $\E\widehat{R}_n(r)$ and $\Var \widehat{R}_n(r)$.
\begin{theorem}\label{esperance2} 
Let $\gamma$ be pairwise interaction potential
defined in  \eqref{modelle} satisfying condition \eqref{RG}.
Let $K_1$ kernel function satisfying Condition  $K(1,1)$. For all $r\in (0,R]$, we have
\[
\lim_{n\rightarrow \infty}\E\widehat{R}_n(r)= \beta^{\star 2}J(r)\exp(-\gamma(r)).
\]
If Condition $K(1,\alpha)$ is satisfied and the function $\exp(-\gamma (r))J(r)$  has bounded and continuous partial derivatives of order $\alpha$ 
in 
$(r-\delta, r+\delta)$ for some $\delta > 0$ and for all $\alpha\geq 1$. Then
\[
\E\widehat{R}_n(r)= \beta^{\star 2}J(r)\exp(-\gamma(r)) +\mathcal{O}(b^\alpha_n) \quad \mbox{as} \ \  n\rightarrow \infty.
\]
\end{theorem}

\begin{theorem} \label{varo}
Let $\gamma$ be pairwise interaction potential
defined in  \eqref{modelle} satisfying condition \eqref{RG}.
Let $K_1$ kernel function satisfying Condition  $K(1,\alpha)$ for all $\alpha\geq 1$ such that
$\int_{\R}K_1^2(\rho) \dd \rho<\infty$ .
For all $r\in (0,R]$, we have,
\[
\lim_{n\rightarrow \infty}{b_n|W_{n\ominus 2R}|}\Var (\widehat{R}_n(r))=
\frac{2\beta^{\star 2}}{\sigma_d r^{d-1}}J(r)\exp(-\gamma(r))\int_{\R}K_1^2(\rho) \dd \rho .
\]
 \end{theorem}  

\subsection{Proof of Theorem \ref{esperance2}}
\begin{proof}
We define
\begin{equation*}
 \widetilde{L}(u_1,...,u_s,\bX)
= \widetilde h(u_1,\bX)... \widetilde h(u_s,\bX), \quad
\widetilde{F}(u_1,...,u_s)= \E[ \widetilde h(u_1,\bX)\cdot \cdot \cdot  \widetilde h(u_s,\bX)]
\end{equation*}
and
$
\widetilde J(||u_1||,...,||u_s||)=1\!\!1(||u_1|| \leq R,...,||u_s||\leq R).
$

The calculation of expectation and variance of $\widehat{R}_n(r)$ is based on the
iterated Georgii-Nguyen-Zessin (GNZ) formula, see  Papangelou \cite{Papangelou}:
{\small
\begin{align}\label{eq:gnz2}
\E\!\!\!\sum_{u_1,...,u_s\in{\bX}}^{\neq}\!\!\!h(u_1,...,u_s, \bX\setminus \{u_1,...,u_s\})
&=  \int\!\!\!...\!\!\!\int \E h(u_1,...,u_s,\bX)\lambda(u_1,...,u_s,\bX)\dd u_1...\dd u_n
\end{align}
}
for non-negative functions $h:(\R^d)^n\times N_{lf} \longrightarrow\R$,
where $\lambda(u_1,...,u_s,\bx)$ is Papangelou conditional intensity and
 is defined (not uniquely) by
\[
\lambda(u_1,...,u_s,\bx)= \lambda(u_1,\bx)\lambda(u_2,\bx\cup \{u_1\})
...\lambda(u_s,\bx\cup \{u_1,...,u_{s-1}\}). 
\]
Applying the preceding formula \eqref{eq:gnz2} for $s=2$, we derive
\begin{align*}
 \E\widehat{R}_n(r)
 &= \frac{1}{b_n|W_{n\ominus 2R}|\sigma_d}\\
&\E \int_{\R^{2d}} \frac{1\!\!1_{W_{n\ominus 2R}}(u)}{||v-u||^{d-1}}\widetilde J(||v-u||)
\widetilde L(u,v,\bX)K_1\bigg(\frac{||v-u||-r}{b_n}\bigg)\lambda(u,v,\bX)\dd u\dd v. 
\end{align*}

For an interaction radius $R$, the Papangelou conditional intensity satisfies 
\[
 \lambda(u,\bx)=\lambda(u,\emptyset) \quad \mbox {for all} \quad \bx \quad \mbox {with} \quad d(u,\bx)>R
\]
since points further  away from $u$ than $R$ do not contribute to the Papangelou conditional intensity at $u$.
Using the finite range property \eqref{RG}, we get
\begin{align*}
 \E\widehat{R}_n(r)
&=\frac{\beta^{\star 2}}{b_n|W_{n\ominus 2R}|\sigma_d}\\
&\E \int_{_{\R^{2d}}}
\frac{1\!\!1_{W_{n\ominus 2R}}(u)}{||v-u||^{d-1}}\widetilde J(||v-u||)
\widetilde L(u,v,\bX)K_1\bigg(\frac{||v-u||-r}{b_n}\bigg)\varPhi(||v-u||)\dd u\dd v \\ 
&=\frac{\beta^{\star 2}}{b_n\sigma_d}\int_{\R^d} \frac{\widetilde J(||s||)}{||s||^{d-1}} \E[\widetilde L(o,s,\bX)]K_1\bigg(\frac{||s||-r}{b_n}\bigg) \varPhi(||s||)\dd s. 
\end{align*}

Recall a property of the integration theory (see Briane and Pagès \cite{M-Briane} or Rudin \cite{W-Rudin}).
Let $\SS^{d-1}$  be the unit sphere in $\R^d$, i.e. 
$
\SS^{d-1}=\{u\in\R^d:||u||=1\},
$
then for any Borel function $f: \R^d\longrightarrow \R_{+}$,
\[
 \int_{\R^d}f(u) \dd u=\int_{0}^{\infty}\int_{\SS^{d-1}}f(rz)r^{d-1}\sigma_d \dd r\dd z.
\]
By combining the above result, we get so:
\begin{align*}
 \E\widehat{R}_n(r)&=\frac{{\beta^\star}^2}{\sigma_d}
\int^{\infty}_{-r/b_n}\int_{\SS^{d-1}}\widetilde J(b_n \varrho+r)\widetilde{F}(o,(b_n \varrho+r)v)
 K_1(\varrho) \varPhi(b_n\varrho+r)\sigma_d\dd \varrho \dd v. 
\end{align*}
With bounded support on the kernel function and by dominated convergence theorem,
we get as $n\rightarrow \infty$, $\E\widehat{R}_n(r)\longrightarrow \beta^{\star 2}J(r)\exp(-\gamma(r))$.
Now, we are going to prove the second part of the Theorem \ref{esperance2}. 
We have a product of two functions  $\widetilde{F}(o,(b_n \varrho+r)v)\varPhi(b_n \varrho+r)$ and we approximate
each one of them with a Taylor formula up to a certain $\alpha$. 
We use Taylor's formula to obtain for $\ n\rightarrow \infty$,
\begin{align*}
\varPhi(b_n \varrho+r)
&=\varPhi(r)+\sum_{k=1}^{\alpha-1}
\frac{(b_n \varrho)^{k}}{k!}\frac{d\varPhi}{dr}(r){dr}
+\frac{b_n^{\alpha}}{\alpha!}\frac{d^\alpha \varPhi}{dr^\alpha}(r+b_n\varrho\theta)
\end{align*}
and
\begin{align*}
\widetilde{F}(o,(b_n \varrho+r)v)
&=\widetilde{F}(o,rv)+\sum_{k=1}^{\alpha-1}
\frac{(b_n \varrho)^{k}}{k!}\frac{d\widetilde{F}(o,rv)}{dr}{dr}
+\frac{b_n^{\alpha}}{\alpha!}\frac{d^\alpha \widetilde{F}}{dr^\alpha}(o,(r+b_n\varrho\theta)v). 
\end{align*}
So we denote this product by $T_n(rv,r)$,
then we have {as} $\ n\rightarrow \infty$
\begin{align*}
 \widetilde{F}(o,(b_n \varrho+r)v)\varPhi(b_n \varrho+r)
 &=\widetilde{F}(o,rv)\varPhi(r)+\sum_{k=1}^{\alpha-1}T_n(rv,r)(b_n\varrho)^k+ \mathcal{O}(b_n^\alpha).
\end{align*}
It follows that,
\begin{align*}
 \E\widehat{R}_n(r)
&={\beta^{\star 2}} J(r)\varPhi(r)\\
&+{\beta^{\star 2}}\int_{\SS^{d-1}}\sum_{k=1}^{\alpha-1}T_n(rv,r)b_n^k\dd v \int_{\R}\varrho^kK_1(\varrho)\dd \varrho\\
&+ \mathcal{O}(b_n^\alpha) \quad \mbox {as} \ n\rightarrow \infty.
\end{align*}
Together with Condition $K(1,\alpha)$ imply the second assertion of Theorem \ref{esperance2}.
\end{proof}
\subsection{Proof of Theorem \ref{varo}}
\begin{proof}
 The proof of Theorem \ref{varo} makes use of the following corollary.
\begin{Cor}\label{varr2}
Consider any  Gibbs point process $\bX$ in $\R^d$ with Papangelou conditional intensity $\lambda$.
For any non-negative, measurable  and symmetric function $f:\R^d\times \R^d\times N_{lf} \longrightarrow \R$, we have
\begin{align*}
& \Var\bigg(\sum_{u,v \in \bX}^{\neq}f(u,v,\bX\backslash\{u,v\})\bigg) \\
&=2\E\int_{\R^{2d}}f^2(u,v,\bX)\lambda(u,v,\bX)\dd u \dd v  \\
&+4\E\int_{\R^{3d}}f(u,v,\bX) f(v,w,\bX) \lambda(u,v,w,\bX)\dd u \dd v \dd w   \\
&+\E\int_{\R^{4d}}f(u,v,\bX)  f(w,y,\bX) \lambda(u,v,w,y,\bX)\dd u \dd v \dd w \dd y  \\
&-\int_{\R^{4d}}\E[f(u,v,\bX)\lambda(u,v,\bX)] \E[f(w,y,\bX) \lambda(w,y,\bX)]\dd u \dd v \dd w \dd y.
\end{align*}
\end{Cor}
\begin{proof} 
Consider the decomposition (see  Jolivet \cite{jolivet2} and Heinrich \cite{Heinrich2})
\begin{align}\label{som}
\bigg(\sum_{u,v \in \bX}^{\neq}f(u,v,\bX\backslash\{u,v\})\bigg)^2
&=2\sum_{u,v \in \bX}^{\neq}f^2(u,v,\bX\backslash\{u,v\})\nonumber\\
&+4\sum_{u,v,w \in \bX}^{\neq}f(u,v,\bX\backslash\{u,v,w\})f(v,w,\bX\backslash\{u,v,w\})\nonumber\\
&+\sum_{u,v,w,y \in \bX}^{\neq}f(u,v,\bX\backslash\{u,v,w,y\})f(w,y,\bX\backslash\{u,v,w,y\}). 
\end{align}
Applying the preceding (GNZ) formula \eqref{eq:gnz2} combining  with \eqref{som}, we obtain
\begin{align*}\label{sommee}
&\Var\sum_{u,v \in \bX}^{\neq}f(u,v,\bX\backslash\{u,v\})\\
&=\E \big(\sum_{u,v \in \bX}^{\neq}f(u,v,\bX\backslash\{u,v\})\big)^2- \big(\E\sum_{u,v \in \bX}^{\neq}f(u,v,\bX\backslash\{u,v\})\big)^2\\
&=2\E\int_{\R^{2d}}f^2(u,v,\bX)\lambda(u,v,\bX)\dd u \dd v  \\
&+4\E\int_{\R^{3d}}f(u,v,\bX) f(v,w,\bX) \lambda(u,v,w,\bX)\dd u \dd v \dd w   \\
&+\E\int_{\R^{4d}}f(u,v,\bX)  f(w,y,\bX) \lambda(u,v,w,y,\bX)\dd u \dd v \dd w \dd y  \\
&-\int_{\R^{4d}}\E[f(u,v,\bX)\lambda(u,v,\bX)] \E[f(w,y,\bX) \lambda(w,y,\bX)]\dd u \dd v \dd w \dd y.
\end{align*}
We obtain the desired result.
\end{proof}
Applying Corollary \ref{varr2} to this function
 \[
 f(u,v,\bX)= 
 \frac{1\!\!1_{W_{n\ominus 2R}}(u)}{||v-u||^{d-1}}\widetilde J(||v-u||) \widetilde L(u,v,\bX)
  K_1\bigg(\frac{||v-u||-r}{b_n}\bigg),
 \]
it is easily seen that 
$
 \Var\widehat{R}_n(r)= A_1+A_2+A_3-A_4,
$
where
\begin{align*}
 A_1
&= \frac{2}{{b^{2}_n|W_{n\ominus 2R}|^2}\sigma_d^2}\\
&\E \int_{\R^{2d}}\frac{1\!\!1_{W_{n\ominus 2R}}(u)}{||v-u||^{2(d-1)}} \widetilde J(||v-u||)\widetilde L(u,v,\bX)
 K_1^2\bigg(\frac{||v-u||-r}{b_n}\bigg)\lambda(u,v,\bX)\dd u\dd v,\\
A_2
&=\frac{4}{b^{2}_n|W_{n\ominus 2R}|^2\sigma_d^2}\E\int_{\R^{3d}}\frac{1\!\!1_{W_{n\ominus 2R}}(u)
1\!\!1_{W_{n\ominus 2R}}(v)}{||v-u||^{d-1}||v-w||^{d-1}}\widetilde J(||v-u||,||v-w||)\widetilde L(u,v,w,\bX)\\
& K_1\bigg(\frac{||v-u||-r}{b_n}\bigg)K_1\bigg(\frac{||v-w||-r}{b_n}\bigg)\lambda
(u,v,w,\bX)\dd u\dd v\dd w,\\
A_3
&=\frac{1}{b^{2}_n|W_{n\ominus 2R}|^2\sigma_d^2}\E\int_{\R^{4d}}
\frac{1\!\!1_{W_{n\ominus 2R}}(u)1\!\!1_{W_{n\ominus 2R}}(w)}{||v-u||^{d-1}||w-y||^{d-1}}\widetilde J(||v-u||,||w-y||) \widetilde L(u,v,w,y,\bX)  \\
&\times   K_1\bigg(\frac{||v-u||-r}{b_n}\bigg) K_1\bigg(\frac{||w-y||-r}{b_n}\bigg)\lambda(u,v,w,y,\bX)\dd u\dd v
 \dd w \dd y,
\end{align*}
and
\begin{align*}
A_4
&=\frac{1}{b^{2}_n|W_{n\ominus 2R}|^2\sigma_d^2}\\ 
&\bigg(\E \int_{\R^{d}}\frac{1\!\!1_{W_{n\ominus 2R}}(u)}{||v-u||^{d-1}} \widetilde J(||v-u||)\widetilde L(u,v,\bX)K_1\bigg(\frac{||v-u||-r}{b_n}\bigg)
\lambda(u,v,\bX)\dd u\dd v\bigg)^2.
\end{align*}
 \paragraph{}
The asymptotic behaviour  of the leading term $A_1$ is obtained by
applying the second order Papangelou conditional intensity given by:
\begin{equation*}
 \lambda(u,v,\bx)= \lambda(u,\bx)\lambda(v,\bx\cup\{u\})
\quad \mbox{for any} \ u, v \in \R^d \quad \mbox{and}\  \bx \in N_{lf}.
\end{equation*}

Using the finite range property \eqref{RG} for each function $ \lambda(u,\bx)$ and
$\lambda(v,\bx\cup\{u\})$, this implies that 
\begin{equation*}
 \lambda(u,\emptyset)=\beta^\star \quad \mbox {and} \quad
\lambda(v,\emptyset\cup\{u\})=\beta^\star \varPhi(||v-u||) \quad \mbox {for all} \quad u,v\in \R^d.
\end{equation*}
And by stationarity of $\bX$, it results

\begin{align*}
A_1
&= \frac{2\beta^{\star 2}}{b^{2}_n|W_{n\ominus 2R}|^2\sigma_d^2} \\
&\int_{\R^{2d}} \frac{1\!\!1_{W_{n\ominus 2R}}(u)}{||v-u||^{2(d-1)}} \widetilde J(||v-u||)\E[\widetilde H(o,v-u,\bX)]K_1^2\bigg(\frac{||v-u||-r}{b_n}\bigg)
\varPhi(||v-u||)\dd u\dd v \\ 
&= \frac{2\beta^{\star 2}}{b_n|W_{n\ominus 2R}|\sigma_d^2} \int^{\infty}_{-r/b_n}\int_{\SS^{d-1}} 
{\frac{\widetilde J(b_n\varrho+r)}{(b_n\varrho+r)^{d-1}}}
{\widetilde{F}(o,(b_n \varrho+r)w)K_1^2(\varrho)}{}\varPhi(b_n\varrho+r) \dd\varrho \dd \sigma (w).
\end{align*}
Dominated convergence theorem and assumption of $K_1$ imply for all $r\in (0, R]$
\[\lim_{n\rightarrow\infty}{b_n|W_{n\ominus 2R}|} A_1= \frac{2\beta^{\star 2}}{\sigma_d r^{d-1}}J(r)\varPhi(r)\int_{\R}K_1^2(\rho) \dd \rho.\]

We will now show that all other integrals to $\Var\widehat{R}_n(r)$  converge to zero.
For the asymptotic behaviour  of the second term $A_2$, we remember the third order Papangelou conditional intensity by
\begin{equation*}\label{uvw}
 \lambda(u,v,w,\bx)
=\lambda(u,\bx)\lambda(v,\bx \cup \{u\}) \lambda(w,\bx\cup\{u,v\})
\end{equation*}
for any $u, v, w \in \R^d$ and  $\bx \in N_{lf}$.
Since $\bX$ is a point process to interact in pairs, the interaction terms due to triplets or higher order 
are equal to one, i.e.
the potential  $\gamma(\by)=0$   when $n(\by) \geq 3$, for $\emptyset\neq \by\subseteq \bx$.
Using the finite range property \eqref{RG} for each function
$\lambda(u,\bx)$, $\lambda(v,\bx \cup \{u\})$ and 
$ \lambda(w,\bx\cup\{u,v\})$ and after a elementary calculation, we have
\begin{equation*}\label{eq3}
\lambda(u,v,w,\emptyset)=
 \left\{
  \begin{array}{rl}
  \beta^{\star3}\varPhi(||v-u||)\varPhi(||w-v||) & \quad\textrm{if}\quad  d(u,w)<R \\
  \beta^{\star3}\varPhi(||v-u||)       & \quad\textrm{otherwise}.
  \end{array}
  \right.
\end{equation*}
Which ensures that $\lambda(u,v,w,\emptyset)$
is a function that depends only variables $||v-u||,||w-v||$, denoted by
$\varPhi_1(||v-u||,||w-v||)$.

According to the stationarity of $\bX$, it follows that
\begin{align*}
A_2
&=\frac{4}{b^{2}_n|W_{n\ominus 2R}|^2\sigma_d^2}
\E \int_{\R^{3d}} \frac{1\!\!1_{W_{n\ominus 2R}}(u)1\!\!1_{W_{n\ominus 2R}}(v)}{||v-u||^{d-1}||v-w||^{d-1}}
 \widetilde J(||v-u||,||v-w||)\widetilde L(u,v,w,\bX)\\
&\times \varPhi_1(||v-u||,||w-v||)K_1\bigg(\frac{||v-u||-r}{b_n}\bigg) K_1\bigg(\frac{||v-w||-r}{b_n}\bigg) \dd u\dd v \dd w \\
&=\frac{4}{|W_{n\ominus 2R}|\sigma_d^2} \int^{\infty}_{-r/b_n} \int^{\infty}_{-r/b_n}\int_{\SS^{d-1}}
\int_{\SS^{d-1}}\frac{|W_{n\ominus 2R}\cap (W_{n\ominus 2R}-(b_n \varrho+r)z)|}{|W_{n\ominus 2R}|}\\
& \times \widetilde{F}(o,(b_n \varrho+r)z,(b_n \varrho^{\prime}+r)z^\prime)
\varPhi_1(b_n \varrho+r,b_n \varrho^{\prime}+r) 
K_1(\varrho)K_1(\varrho^\prime)\dd \varrho \dd \varrho^\prime \dd \sigma_d(z) \dd \sigma_d(z^{\prime}). 
\end{align*}
The asymptotic behaviour of the leading term $A_2$ is obtained by applying the dominated convergence theorem.
When multiplied by $b_nW_{n\ominus 2R}|$, we get
$\lim_{n\rightarrow\infty}{b_n|W_{n\ominus 2R}|} A_2=0$.

Next we introduce the finite range property \eqref{RG} and reasoning analogous with the foregoing on
$\lambda(u,v,w,y,\emptyset)$.  which ensures that $\lambda(u,v,w,y,\emptyset)$ is a 
function that depends only variables $||v-u||,||y-w ||,||w-u ||,||w-v||)$, denoted by 
$\varPhi_2(||v-u||,||y-w ||,||w-u ||,||w-v||)$. We find that
{\small
\begin{align*}
A_3
&=\frac{1}{b^{2}_n|W_{n\ominus 2R}|^2\sigma_d^2}\E\int_{\R^4}
\frac{1\!\!1_{W_{n\ominus 2R}}(u)1\!\!1_{W_{n\ominus 2R}}(w)}{||v-u||^{d-1}||w-y||^{d-1}}\widetilde J(||v-u||,||w-y||)\widetilde L(u,v,w,y,\bX)  \\
&\times  K_1\bigg(\frac{||v-u||-r}{b_n}\bigg)K_1\bigg(\frac{||w-y||-r}{b_n}\bigg)
\lambda(u,v,w,y,\bX)\dd u\dd v\dd w\dd y\\
&=\frac{1}{|W_{n\ominus 2R}|\sigma_d^2} 
 \int_{\R^{d}} \int^{\infty}_{-r/b_n}\int^{\infty}_{-r/b_n}\int_{\SS^{d-1}} \int_{\SS^{d-1}}  \frac{|W_{n\ominus 2R}\cap (W_{n\ominus 2R}-w) |}
{|W_{n\ominus 2R}|}K_1(\varrho)K_1(\varrho^\prime)\\
&\times \varPhi^\star_2(b_n\varrho+r,b_n\varrho^\prime+r,||w||,||(b_n\varrho+r)z-w||)\dd w \dd \varrho \dd\varrho^\prime \dd \sigma_d(z) \dd \sigma_d(z^\prime).
\end{align*}
}
Where

$
\varPhi^\star_2(b_n\varrho+r,b_n\varrho^\prime+r,||w||,||(b_n\varrho+r)z-w||)
=
\widetilde J(b_n\varrho+r,b_n\varrho^\prime+r)\widetilde{F}(o,(b_n\varrho+r)z,(b_n\varrho^\prime+r)z^\prime)\varPhi_2(b_n\varrho+r,b_n\varrho^\prime+r,||w||,||(b_n\varrho+r)z-w||)
.$

By dominated convergence theorem, we get $\lim_{n\rightarrow\infty}{b_n|W_{n\ominus 2R}|} A_3=0$.

For asymptotic behaviour of the leading term  $A_4$, 
it then suffices to repeat the arguments developed previously to conclude the following result.
{\small
\begin{align*}
A_4
&= \frac{\beta^{\star4}}{b^{2}_n|W_{n\ominus 2R}|^2\sigma_d^2} \\
& \int_{\R^{4d}}1\!\!1_{W_{n\ominus 2R}}(u)1\!\!1_{W_{n\ominus 2R}}(w) \frac{\widetilde J(||v-u||,||w-y||)}{||v-u||^{d-1}||w-y||^{d-1}}
\E[\widetilde L(u,v,\bX)]\E[\widetilde L(w,y,\bX)]\\
&\times \varPhi(||v-u||) \varPhi(||y-w||) K_1\bigg(\frac{||v-u||-r}{b_n}\bigg)K_1\bigg(\frac{||w-y||-r}{b_n}\bigg)\dd u\dd v \dd w\dd y\\ 
&= \frac{\beta^4}{b^{2}_n|W_{n\ominus 2R}|\sigma_d^2}\\
&\int_{\R^{3d}} \frac{|W_{n\ominus 2R} \cap (W_{n\ominus 2R}-w)|}{|W_{n\ominus 2R}|||v-u||^{d-1}||w-y||^{d-1}} \widetilde J(||v||,||w-y||)\E[\widetilde H(o,v,\bX)] \E[\widetilde L(w,y,\bX)]\\
&\times \varPhi(||v||)\varPhi(||y-w||) K_1\bigg(\frac{||v||-r}{b_n}\bigg)K_1\bigg(\frac{||w-y||-r}{b_n}\bigg)\dd v \dd w\dd y\\ 
&=\frac{\beta^4}{|W_{n\ominus 2R}|\sigma_d^2} \\
&\times \int_{\R^{d}} \int^{\infty}_{-r/b_n}\int^{\infty}_{-r/b_n} \int_{\SS^{d-1}} \int_{\SS^{d-1}}\frac{|W_{n\ominus 2R}\cap (W_{n\ominus 2R}-w)|}{|W_{n\ominus 2R}|}
  \widetilde J(b_n \varrho+r,b_n \varrho^\prime+r)\widetilde{F}(o,(b_n \varrho+r)z) \\
 & \times \widetilde{F}o,(b_n \varrho^\prime+r)z^\prime) \varPhi(b_n\varrho+r) \varPhi(b_n \varrho^\prime+r)
  K_1(\varrho)K_1(\varrho^\prime)\dd w \dd \varrho \dd \varrho^\prime \dd \sigma_d(z) \dd \sigma_d(z^\prime).
\end{align*}
}
Then by dominated convergence theorem, we get $\lim_{n\rightarrow\infty}{b_n|W_{n\ominus 2R}|} A_4=0$.
\end{proof}
\section{Strong consistency}\label{rates}
\subsection{Rates uniform strong convergence of the kernel-type estimator}
Before realizing the strong consistency $\widehat{\varPhi}_n(r)$ we introduce some necessary definitions and notation.
A Young function $\psi$ is a real convex nondecreasing function
defined on $\R^{+}$ which satisfies
$\lim_{t\to\infty}\psi(t)=+\infty$ and $\psi(0)=0$. We define the
Orlicz space $L_{\psi}$ as the space of real random variables $Z$
defined on the probability space $(N_{lf}, \F, \P)$ such that
$E[\psi(\vert Z\vert/c)]<+\infty$ for some $c>0$. The Orlicz space
$L_{\psi}$ equipped with the so-called Luxemburg norm $\| .\|_{\psi}$ defined for any real random variable $Z$ by
\[
\| Z\|_{\psi}=\inf\{\,c>0\,;\,E[\psi(\vert Z\vert/c)]\leq 1\,\}
\]
is a Banach space. For more about Young functions and Orlicz
spaces one can refer to Krasnosel'skii and Rutickii \cite{K.R61}.
Let $\theta>0$. We denote by $\psi_{\theta}$ the Young function
defined for any $x\in\R^{+}$ by
\[
\psi_{\theta}(x)=\exp((x+\xi_{\theta})^{\theta})-\exp(\xi_{\theta}^{\theta})\quad\textrm{where}\quad
\xi_{\theta}=((1-\theta)/\theta)^{1/\theta}1\!\!1{\{0<\theta<1\}}.
\]
On the lattice $\Z^{d}$ we define the lexicographic order as
follows: if $i=(i_{1},...,i_{d})$ and $j=(j_{1},...,j_{d})$ are
distinct elements of $\Z^{d}$, the notation $i<_{lex}j$ means that
either $i_{1}<j_{1}$ or for some $p$ in $\{2,3,...,d\}$,
$i_{p}<j_{p}$ and $i_{q}=j_{q}$ for $1\leq q<p$. Let the sets
$\{V_{i}^{k}\,;\,i\in\Z^{d}\,,\,k\in\N^{\ast}\}$ be defined as
follows:
\[
V_{i}^{1}=\{j\in\Z^{d}\,;\,j<_{lex}i\},
\]
and for $k\geq 2$
\[
V_{i}^{k}=V_{i}^{1}\cap\{j\in\Z^{d}\,;\,\vert i-j\vert\geq
k\}\quad\textrm{where}\quad \vert i-j\vert=\max_{1\leq l\leq
d}\vert i_{l}-j_{l}\vert.
\]
For any subset $\Gamma$ of $\Z^{d}$ define
$\F_{\Gamma}=\sigma(\varepsilon_{i}\,;\,i\in\Gamma)$ and set
\[
E_{\vert
k\vert}(\varepsilon_{i})=E(\varepsilon_{i}\vert\F_{V_{i}^{\vert
k\vert}}),\quad k\in V_{i}^{1}.
\]
Denote $\theta(q)=2q/(2-q)$ for $0<q<2$ and by convention $1/\theta(2)=0$.
\paragraph{} 
Next we list a set of conditions which are needed to obtain (rates of) uniform strong consistency over some 
compact set $[r_1,r_2]$ in $(0,R]$ of the estimator $\widehat{R}_n(r)$ to the function 
 $\beta^{\star2}J(r)\varPhi(r)$. The following assumption is imposed:
   
\textbf{Condition $\Lp:$}
The kernel function $K$ is a Lipschitz condition, i.e. there exists a constant $\eta> 0$ such that
\[
\big|K_1(\rho)-K_1(\rho^\prime)\big|\leq \eta|\rho-\rho^\prime|  \quad\textrm{for all}\quad  \rho,\rho^\prime \in [r_1,r_2].
\]
\paragraph{}
Strong uniform consistency for the resulting estimator are obtained via assumptions of belonging
to  Orlicz spaces induced by exponential Young functions for stationary real random fields which allows us to 
derive the  Kahane-Khintchine inequalities by El Machkouri \cite{Machkouri02}. 
 Our results also carry through the most important particular case of Orlicz spaces random fields,
 we use the inequality follows from a Marcinkiewicz-Zygmund type inequality by Dedecker \cite{Dedecker01}.

Now, we split up the sampling window $W_{n\ominus 2R}$  into cubes such as
$W_{n\ominus 2R}=\displaystyle \cup_{i\in \Gamma_n}\Lambda_i$, where $\Lambda_i$ are centered at $i$
and assume that $\Gamma_n=\{-n,...,0,...,n\}^d$ increases towards $\Z^d$. We split up  $\widehat{R}_n(r)$ as follows:
\begin{equation*}\label{range24}
 \widehat{R}_n(r)= 
\frac{1}{b_n|W_{n\ominus 2R}|\sigma_d}\sum_{i\in \Gamma_n}R_k(r)
\end{equation*}
\begin{equation*}\label{eeqiso}
  R_k(r)=  \sum_{\begin{subarray}{c} u,v \in \bX \\ ||v-u||\leq R.
    \end{subarray}} 
 ^{\neq}\frac{1\!\!1_{\Lambda_k}(u)}{||v-u||^{d-1}}\widetilde h(u,\bX \setminus \{u,v\})
\widetilde h(v,\bX\setminus \{u,v\})K_1\bigg(\frac{||v-u||-r}{b_n}\bigg).
 \end{equation*}
Note for all  $k\in \Gamma_n$,
$
 \bar{R}_k=R_k(r)-\E R_k(r) \quad \mbox{and} \quad S_n=\sum_{k\in \Gamma_n}\bar{R}_k(r).
$
\begin{theorem}\label{consiso}
Under Conditions $K(1,\alpha)$ and $\Lp$. Further, assume that $J(r)\exp(-\gamma(r))$ has bounded 
and continuous partial derivatives of order $\alpha$ in $[r_1-\delta,r_2+\delta]$ for some $\delta>0$. 
\begin{fleuve}
 \item If there exists $0<q<2$ such that $\bar{R}_0\in\sL_{\psi_{\theta(q)}}$ and
\begin{equation}\label{centrage3}
\sum_{k\in V_{0}^{1}} \left\|\sqrt{\big\vert \bar{R}_kE_{\vert k\vert}(\bar{R}_0)\big\vert}\right\|_{\psi_{\theta(q)}}^{2}<\infty.
\end{equation}
Then
\[
\sup_{r_1\leq r\leq r_2} \big|\widehat{R}_n(r)-\beta^{\star2} J(r)\exp(-\gamma(r))\big|=\mathcal{O}_{a.s.}
\bigg(\frac{(\log n)^{1/q}}{b_n{n}^{d/2}}\bigg)+\mathcal{O}(b_n^\alpha)  \quad\textrm{as}\quad
 \ n\rightarrow \infty.
\]  
\item If $\bar{R}_0\in\sL^\infty$ and
\begin{equation}\label{centrage31}
\sum_{k\in V_{0}^{1}} \left\| \bar{R}_kE_{\vert k\vert}(\bar{R}_0)\right\|_{\infty}<\infty.
\end{equation}
Then
  \[
\sup_{r_1\leq r\leq r_2} \big|\widehat{R}_n(r)-\beta^{\star2} J(r)\exp(-\gamma(r))\big|=\mathcal{O}_{a.s.}
\bigg(\frac{(\log n)^{1/2}}{b_n{n}^{d/2}}\bigg)+\mathcal{O}(b_n^\alpha)  \quad\textrm{as}\quad
 \ n\rightarrow  \infty.
\] 
\item If there exists $p>2$ such that $\bar{R}_0\in \sL^p$ and
\begin{equation}\label{centrage32}
\sum_{k\in V_{0}^{1}} \left\| \bar{R}_kE_{\vert k\vert}(\bar{R}_0)\right\|_{\frac{p}{2}}<\infty.
\end{equation}
Assume that $b_{n}=n^{-q_{2}}(\log n)^{q_{1}}$ for some $q_{1},q_{2}>0$. Let $a,b\geq 0$ be fixed and if 
$
a(p+1)-d^2/2-q_2>1 \quad\textrm{and}\quad b(p+1)+q_1 >1.
$
Then
\[
\sup_{r_1\leq r\leq r_2}\big|\widehat{R}_n(r)-\beta^{\star2} J(r)\exp(-\gamma(r))\big|=\mathcal{O}_{a.s.}
\bigg(\frac{n^{a}(\log n)^{b}}{b_{n}{n}^{d/2}}\bigg)+\mathcal{O}(b_n^\alpha)
 \quad\textrm{as}\quad \ n\rightarrow  \infty.
\]
\end{fleuve}
\end{theorem}
 \begin{remarque}\label{condition55}
From the Markov property of $\bX$ entails that for $i \neq 0 $ are not neighborhoods, then 
  $\bar{R}_i$ et $\bar{R}_o$ are conditionally independent, i.e
$\E[\bar{R}_0 |(X_{\Lambda_i} ; i \neq 0 ] = 0$.
 Since $\sigma(R_i,i\in V_0^k)$ is contained in $\sigma(X_{\Lambda_i},i\neq 0)$ for $k>l$, for some integer $l$, 
it follows immediately that conditions 
 \eqref{centrage3}, \eqref{centrage31}, \eqref{centrage32} are satisfied.
\end{remarque}
\subsection{Proof of Theorem \ref{consiso}}
\begin{proof}
 To establish rates of the uniform $\P $- a.s. convergence for the
 estimator $\widehat{R}_n(r)$, we apply a triangle inequality decomposition allows for
\begin{align*}\label{BB23}
\sup_{s_{i-1}\leq r\leq s_i}\bigg|\widehat{R}_n(r)-\E\widehat{R}_n(r)|
&\leq \sup_{s_{i-1}\leq r,\rho\leq s_i}\bigg| \widehat{R}_n(r)-\widehat{R}_n(\rho)\bigg|\\
&+ \sup_{s_{i-1}\leq r,\rho\leq s_i}\bigg| \E\widehat{R}_n(r)-\E\widehat{R}_n(\rho)\bigg|\\
&+ \sup_{s_{i-1}\leq \rho\leq s_i}\bigg| \widehat{R}_n(\rho)-\E\widehat{R}_n(\rho)\bigg|.
\end{align*}
The compact set $[r_1,r_2 ]$ is covered by the intervals $C_i=[s_{i-1}-s_i]$, where  $s_i=r_1+i(r_2-r_1)/N, i=1,...,N$.
Choosing $N$ as the largest integer satisfying $N\leq c/l_n$ and $l_n=r_nb_n^{2}$.
Under the condition $\Lp$, we deduce that there exists a constant $\eta> $ 0 such that 
 for any $n$ sufficiently large

\begin{align*}
\sup_{s_{i-1}\leq r,\rho\leq s_i}\bigg| \widehat{R}_n(r)-\widehat{R}_n(\rho)\bigg|
&\leq \frac{1}{b_n^2}\eta|r-\rho|\widetilde{R}_n\\
&\leq r_n\widetilde{R}_n
\end{align*}
where
\begin{equation*}
 \widetilde{R}_n=\frac{1}{\sigma_d|W_{n\ominus 2R}|}\sum_{\begin{subarray}{c} u,v \in \bX \\ ||v-u||\leq R
   \end{subarray}} 
^{\neq}\frac{1\!\!1_{W_{n\ominus 2R}}(u)}{||v-u||^{d-1}}\widetilde h(u,\bX \setminus \{u,v\})\widetilde h(v,\bX\setminus \{u,v\}). 
\end{equation*}
Follows from the last inequalities and the Nguyen and Zessin ergodic theorem \cite{A-NguZes79}: 
\[ \sup_{s_{i-1}\leq r,\rho\leq s_i}\bigg| \widehat{R}_n(r)-\widehat{R}_n(\rho)\bigg|=\mathcal{O}_{p.s.}(r_n)
\quad\textrm{as}\quad \ n\rightarrow  \infty.
\]
As well
\begin{equation*}
\sup_{s_{i-1}\leq r,\rho\leq s_i}\bigg| \E\widehat{R}_n(r)-\E\widehat{R}_n(\rho)\bigg|=
\mathcal{O}_{p.s.}(r_n) \quad\textrm{as}\quad \ n\rightarrow  \infty.
\end{equation*}
\begin{lemme}\label{compact5}
Assume that either \eqref{centrage3} holds for some  $0<q< 2$  such that
$\bar{R}_0\in\sL_{\psi_{\theta(q)}}$ and $r_n=(\log n)^{1/q}/b_n(\sqrt{n})^{d}$ 
or \eqref{centrage31} holds such that  $\bar{R}_0\in\sL_{\infty}$ and $r_n=(\log n)^{1/2}/b_n(\sqrt{n})^{d}$.
Then 
\[
\sup_{s_{i-1}\leq \rho\leq s_i}\bigg| \widehat{R}_n(\rho)-\E\widehat{R}_n(\rho)\bigg|=\mathcal{O}_{p.s.}(r_n)
\quad\textrm{as}\quad \ n\rightarrow  \infty.
\]
\end{lemme}
\begin{proof}
For $\varepsilon>0$, using Markov's inequality, we get
\hspace*{-18cm}
 {\small
\begin{align*}
& \P\left(\vert\widehat{R}_n(r)-\E\widehat{R}_n(r)\vert >\varepsilon r_n\right)
= \P\left(\vert S_{n}\vert>\varepsilon r_{n}b_nn^d\right)\nonumber\\
&\leq \exp\bigg[-\left(\frac{\varepsilon\,r_nb_nn^d}{||S_{n}||_{\psi_{\theta(q)}}}+\xi_{q}\right)^{q}\bigg]
\E\exp\bigg[ \left(\frac{|S_n|}{||S_{n}||_{\psi_{\theta(q)}}}+\xi_{q}\right)^q \bigg].
\end{align*}}

Therefore, we assume that there exists a real $0<q<2$, such that  $\bar{R}_0 \in \sL_{\psi_{\theta(q)}}$
 and using Kahane-Khintchine inequalities (cf. El Machkouri \cite{Machkouri02}, Theorem 1), we have
\hspace*{-18cm}
 {\small
 \begin{align*}
 \P\left(\vert\widehat{R}_n(r)-\E\widehat{R}_n(r)\vert >\varepsilon r_n\right)
&\leq\P\left(\vert S_{n}\vert>\varepsilon r_{n}b_nn^d\right)\nonumber\\
&\leq (1+e^{\xi_{q}^{q}})\,\exp\bigg[-\left(\frac{\varepsilon\,r_nb_nn^d}
{M(\sum_{i\in \Gamma_n}b_{i,q}(\bar{R}))^{1/2}}+\xi_{q}\right)^{q}\bigg]
\end{align*}}
denote
\begin{equation*}\label{biq}
b_{i,q}(\bar{R})=\big\|\bar R_0\big\|_{\psi_{\theta(q)}}^{2}+\sum_{k\in V_{0}^{1}}\left\|\sqrt{\big\vert \bar{R}_kE_{\vert k\vert}(\bar{R}_0)\big\vert}\right\|_{\psi_{\theta(q)}}^{2}.
\end{equation*}
We derive that if condition \eqref{centrage3} holds, then there exist
constant $C>0$ and so if $r_n=(\log n)^{1/q}/b_n(\sqrt{n})^{d}$
\begin{equation*}\label{unif1}
 \sup_{r_1\leq r\leq r_2}\P(\vert\widehat{R}_n(r)-\E\widehat{R}_n(r)\vert >\varepsilon r_n)
\leq(1+e^{\xi_{q}^{q}})\,\exp\bigg[-\frac{\,\varepsilon^{q}\,\log
n} {C^{q}}\bigg].
\end{equation*}

Now, we will accomplish the second step the proof of Proposition \ref{consiso}.
 Using Kahane-Khintchine inequalities (cf. El Machkouri \cite{Machkouri02}, Theorem 1)
with $q=2$, such that $\bar{R}_0\in\sL_{\infty}$, we have
\begin{equation*}\label{ed}
 \P\left(\vert\widehat{R}_n(r)-\E\widehat{R}_n(r)\vert >\varepsilon r_n\right)
\leq 2\,\exp\bigg[-\left(\frac{\varepsilon\,r_nb_nn^d}
{M(\sum_{i\in \Gamma_n}b_{i,2}(\bar R))^{1/2}}\right)^{2}\bigg]
\end{equation*}
denote
\begin{equation*}\label{bi2}
b_{i,2}(\bar R)=\big\|\bar{R}_0\big\|_{\infty}^{2}+\sum_{k\in V_{0}^{1}}\big\| \bar{R}_{k}E_{\vert k\vert}(\bar{R}_0)\big\|_{\infty}.
\end{equation*}
We derive that if condition \eqref{centrage31} holds and so if $r_n=(\log n)^{1/2}/b_n(\sqrt{n})^{d}$, there exists 
$C>0$ such that
\begin{equation*}\label{unif2}
 \sup_{r_1\leq r\leq r_2}\P(\vert\widehat{R}_n(r)-\E\widehat{R}_n(r)\vert >\varepsilon r_n)
\leq 2\,\exp\bigg[-\frac{\,\varepsilon^{2}\,\log
n} {C^{2}}\bigg].
\end{equation*}
choosing $\varepsilon$ sufficiently large, 
therefore, it follows with Borel-Cantelli's lemma
\[
\P(\lim\sup_{n\rightarrow\infty}\sup_{s_{i-1}\leq \rho\leq s_i}\bigg| \widehat{R}_n(\rho)-\E\widehat{R}_n(\rho)\bigg|>\varepsilon r_n)=0.
\]
\end{proof}
\paragraph{}
Now, we will accomplish the last step the proof of Theorem \ref{consiso}.
\begin{lemme}
Assume \eqref{centrage32} holds for some $p>2$  such that $\bar{R}_0\in\sL^p$ and
$b_{n}=n^{-q_{2}}(\log n)^{q_{1}}$ for some constants
$q_{1},q_{2}>0$. Let  Let  $a,b\geq 0$ be fixed and  denote $r_n=n^a (\log n)^b/b_n(\sqrt{n})^d$.
If
\[a(p+1)-d/2-q_2>1 \quad\textrm{et}\quad b(p+1)+q_1 >1,\]
then
\[
 \sup_{s_{i-1}\leq \rho\leq s_i}\bigg| \widehat{R}_n(\rho)-\E\widehat{R}_n(\rho)\bigg|=\mathcal{O}_{p.s}(r_n) \quad\textrm{as}\quad \ n\rightarrow  \infty.
\]
\end{lemme}
\begin{proof}
Let  $p>2$ be fixed, such that $\bar{R}_0\in \sL^p$ and for any $\varepsilon>0$,
\begin{align*}
 \P(\vert\widehat{R}_n(r)-\E\widehat{R}_n(r)\vert >\varepsilon r_n)
&=\P\left(\vert S_{n}\vert>\varepsilon r_{n}b_nn^d\right)\\
&\leq\frac{\varepsilon^{-p}\E\vert S_{n}\vert^{p}}{r_{n}^{p}b_n^pn^{pd}}\\
&\leq\frac{\varepsilon^{-p}}{r_{n}^{p}b_n^pn^{pd}}\left(2p\sum_{i\in \Gamma_n}c_i(\bar{R})\right)^{p/2}.
\end{align*}
The last inequality follows from a Marcinkiewicz-Zygmund type inequality by Dedecker \cite{Dedecker01}, where
\begin{equation*}\label{ddd}
c_{i}(\bar{R})=\|\bar{R}_{i}\|_{{p}}^2+\sum_{k\in V_{i}^{1}}\|\bar{R}_{k}E_{\vert k-i\vert}(\bar{R}_{i})\|_{\frac{p}{2}}.
 \end{equation*}
Under assumption \eqref{centrage32} and with the stationarity of $\bX$, we derive that there exists $C>0$ such that
\begin{align*}
\P \left(\sup_{s_{i-1}\leq \rho\leq s_i}\bigg| \widehat{R}_n(\rho)-\E\widehat{R}_n(\rho)\bigg|>\varepsilon r_{n}\right)
&\leq N
\sup_{r_1\leq r\leq r_2}\P(\vert\widehat{R}_n(r)-\E\widehat{R}_n(r)\vert >\varepsilon r_n)\\
&\leq N\frac{\kappa\varepsilon^{-p}}{r_{n}^{p}b_n^p(\sqrt{n})^{pd}}.
\end{align*}
As $N\leq c/l_n$ and $l_n=r_nb_n^{2}$, then
for  $r_n=n^a (\log n)^b/b_n(\sqrt{n})^d$ , it results for $n$ $\varepsilon$ sufficiently large,
\begin{align*}
\P\left( \sup_{s_{i-1}\leq \rho\leq s_i}\bigg| \widehat{R}_n(\rho)-\E\widehat{R}_n(\rho)\bigg|>\varepsilon r_{n}\right)
&\leq \frac{\kappa\varepsilon^{-p}}{n^{a(p+1)-d/2}(\log n)^{b(p+1)}b_n}\\
&\leq \frac{\kappa\varepsilon^{-p}}{n^{a(p+1)-d/2-q_2}(\log n)^{b(p+1)+q_1}}.
\end{align*}
For $a(p+1)-d/2-q_{2}>1$ et ${b(p+1)+q_1}>1$, we get for any $\varepsilon>0$ $\varepsilon >0$
\[\sum_{n\geq 1}\P\left( \sup_{s_{i-1}\leq \rho\leq s_i}\bigg| \widehat{R}_n(\rho)-\E\widehat{R}_n(\rho)\bigg|>\varepsilon r_{n}\right)<\infty.\]
\end{proof}
Considering these arguments the proofs of Theorem \ref{consiso} are completed,
it results from a direct application of the theorem of Borel-Cantelli and by Theorem \ref{esperance2}
we have
\begin{equation*}\label{B65}
\sup_{r_1\leq r\leq r_2}|\E\widehat{R}_n(r)-{R}(r)|=\mathcal{O}(b_n^\alpha) \quad\textrm{as}\quad
 \ n\rightarrow  \infty.
\end{equation*}
\end{proof}
\textbf{Acknowledgments}

The research was supported by laboratory Jean Kuntzmann, Grenoble University, France.
\bibliographystyle{plain}
\bibliography{bibliographie}

\end{document}